\documentclass[CCO,PDF,12pt]{cco_author} % use PDF command to enable PDFLaTeX driver
\usepackage{layout}

\title{Iterated line graphs with only negative eigenvalues $-2$, their complements and energy}
\shorttitle{Iterated line graphs with only negative eigenvalues $-2$, their complements and energy}
\author{Harishchandra S. Ramane\inst{1}, B. Parvathalu\inst{2}\footnote{Corresponding Author},
        K. Ashoka\inst{3}, Daneshwari Patil\inst{1}
       }

\shortauthor{Harishchandra S. Ramane, B. Parvathalu, K. Ashoka, Daneshwari Patil}

\institute{\centering \inst{1}
           Department of Mathematics, Karnatak University, Dharwad 580003, Karnataka, India \\
           {\tt hsramane@yahoo.com, daneshwarip@gmail.com}
           \inst{2}
           Department of Mathematics, Karnatak University's Karnatak Science/Arts College  Dharwad  580001, Karnataka, India \\
           {\tt bparvathalu@gmail.com}
           \inst{3}
           Department of Mathematics, Christ University, Bangalore 560029, Karnataka, India\\
           {\tt ashokagonal@gmail.com}
          }
\abstract{The graphs with all equal negative or positive eigenvalues are special kind in the spectral graph theory.  In this article, several iterated line graphs $\mathcal{L}^k(G)$ with all equal negative eigenvalues $-2$ are characterized for $k\ge 1$ and their energy consequences are presented. Also the spectra and the energy of complement of these graphs are obtained, interestingly they have exactly two positive eigenvalues with different multiplicities. Moreover, we characterize a large class of equienergetic graphs which extend the existing results.
}

\keywords{Energy of a graph, Eigenvalues of graphs, Equitable partition, Iterated line graphs, Complement of a graph}
\history{Received: ....;
 Revised:   .....;
 Accepted: ....;
Published Online: ....
}% will be inserted by the publisher
\msc{05C50, 05C76, 05C12}
\begin{document}
\maketitle

\section{Introduction}\label{sec1}
Spectral graph theory is aimed at answering the questions related to the graph theory in terms of eigenvalues of matrices which are naturally associated with graphs. Graphs with all equal positive or negative eigenvalues are very special kind. The line graph of a graph has a special property that its least eigenvalue is not smaller than $-2$ \cite{Hoffman}. In the last two decades the graphs with least eigenvalues $-2$ have been well studied \cite{Cvetkovic_Sgen,Cvetten_yrs,VijayaKumar}. The problem of particular interest in this class of graphs is the graphs with all negative eigenvalues equal to $-2$. This problem can be restated in terms of the eigenvalues of signless Laplacian matrix Q for line graphs as the graphs with all its signless Laplacian eigenvalues belongs to the set $[2, \infty)\cup\{0\}$ with the help of the relation \eqref{line_eig_snles_eqn}. In \cite{HSRamane} Ramane et al. obtained the spectra and the energy of iterated regular line graphs $\mathcal{L}^k(G)$ for $k\ge 2$  and thus gave infinitely many pairs of non-trivial equienergetic graphs which belong to the above class of graphs. Let $\rho$  be the property that a graph $G$  has all its negative eigenvalues equal to $-2$. In this paper we are motivated to find several classes of iterated line graphs $\mathcal{L}^k(G)$ with the property $\rho$ for $k\ge 1$, spectra of their complements and energy consequences. As a consequence of the energy of these graphs we present a large class of equienergetic graphs which generalize the existing results. The energy relation between a graph $G$ and its complement were studied in \cite{Mojallal,Nikiforov,ramathesis,Ramane_Bp,Ramane-2023}, we extend the results pertaining $\mathcal{E}(\overline{G}) = \mathcal{E}(G)$
to non-regular iterated line graphs with the property $\rho$. The energy of line graphs is well studied in \cite{Kinkar_Das,Gutman_line,Hyper,HSRamane,OscarRojo}. In this paper, we also present hyperenergetic iterated line graphs and their complements.  The software sage \cite{sage} is used to verify some of the results.\\
This paper is organized as follows: Section $2$ introduces basic definitions, equitable partitions and known results on the spectra and energy of graphs. In Section $3$, it is proven that the two distinct quotient matrices defined for an equitable partition of the $H$-join of regular graphs produces the same partial spectrum for the adjacency matrix, Laplacian matrix and signless Laplacian matrix. Section $4$ presents findings on the spectra and energy of iterated line graphs that satisfy the property $\rho$. Section $5$ discusses the spectra and energy of the complements of the graphs covered in Section $4$. Section $6$ provides an upper bound for the independence number of iterated line graphs and their complements along with a theorem regarding the minimum order of a connected graph whose complement is also connected and considers line graphs satisfying the property $\rho$.
%%%%%%%%%%%%%%%%%%%%%%%%%%%%%%%%%%%%%%%%%%%%%%%%%%%%%%
\section{Preliminaries}
In this paper, simple, undirected and finite graph $G$ are considered with vertex set $V(G)=\{v_1, v_2, \ldots, v_n\}$ and edge set $E(G)=\{e_1, e_2, \ldots, e_m\}$. Let $e=v_iv_j$ be an edge of $G$ with its end vertices $v_i$ and $v_j$. The order and size of a graph $G$ refer to the number of vertices and the number of edges, respectively. The complement $\overline{G}$ of a graph $G$ has same vertices as in $G$ but two vertices are adjacent in $\overline{G}$ if and only if they are not adjacent in $G$. Let the subgraph of a graph $G$ obtained by deleting vertices $v_1, v_2, \ldots, v_k$, $k<n$ and the edges incident to them in $G$ be $G-\{v_1, v_2, \ldots, v_k\}$ and simply $G-v$ if one vertex $v$ and edges incident to it are deleted. The {\bf adjacency matrix} $A(G)$ or simply $A$ of a graph $G$ of order $n$ is the $n\times n$ matrix indexed by $V(G)$ whose $(i, j)$-th entry is defined as $a_{ij} = 1$ if $v_iv_j\in E(G)$ and $0$ otherwise. The \textbf{Laplacian} $L$ and \textbf{signless Laplacian} $Q$ of a graph are defined as $L = D - A$ and $Q = D + A$, respectively, where $D = [d_i]$ is the diagonal degree matrix of appropriate order, and the $i$-th diagonal entry $d_i$ represents the degree of vertex $v_i$.
The matrix 
$Q$ is positive semidefinite and possesses real eigenvalues. The spectrum of a square matrix $M$ is the multiset of its eigenvalues, including their algebraic multiplicities. Denote the  characteristic polynomial and the spectrum of a matrix $M$ associated to a graph $G$, respectively by $\varphi(M(G);x)$ and $Sp_M(G)$. Let $a Sp_M(G)\pm b = \{a\lambda\pm b: \lambda\in Sp_M(G)\}$ for any real numbers $a$ and $b$. Let the difference between two sets $A$ and $B$ is denoted by $A \setminus B$. The spectrum of a graph is the spectrum of its adjacency matrix $A$. The inertia of a graph $G$ is the triplet $(n^+, n^-, n^0)$, representing the number of positive, negative and zero eigenvalues, respectively. 
Given a graph $G$ we denote the spectrum of adjacency matrix and signless Laplacian matrix respectively by $Sp_A(G)=\{\lambda_1^{m_1}, \lambda_2^{m_2}, \ldots, \lambda_k^{m_k}\}$ and $Sp_Q(G)=\{q_1^{m_1}, q_2^{m_2}, \ldots, q_k^{m_k}\}$, where $\lambda_{i}$'s and $q_{i}$'s are indexed in descending order, and $m_i$ is the multiplicity of the respective eigenvalue for $1\le i\le k$. Denote the least eigenvalue of signless Laplacian by $q_{\min}$. The {\bf energy} \cite{Gutman} of a graph $G$  is defined as $\mathcal{E}(G) = \sum\limits_{i=1}^{n}\lvert\lambda_i\rvert = 2\sum\limits_{i=1}^{n^+}\lambda_{i} = -2\sum\limits_{i=1}^{n^-}\lambda_{n-i+1}$. Two graphs $G_1$ and $G_2$ having the same order are called {\bf equienergetic} if they share the same energy.  \noindent A set of vertices in a graph $G$ is independent if no two vertices in the set are adjacent. The {\bf independence number} $\alpha{(G)}$ of a graph $G$ is the maximum cardinality of an independent set of $G$. As usual the graphs $C_n, K_n$ and $P_n$ denote the cycle graph, complete graph and path graph on $n$ vertices respectively. For additional notation and terminology,  we refer \cite{Cvetk_book}.

\begin{definition} \cite{Turandef}
	The {\bf Tur\'{a}n graph} $T_r(n)$, $r\ge 1$, is the complete $r$-partite graph of order $n$ with all parts of size either $\lfloor n/r \rfloor$ or $\lceil n/r \rceil$.
\end{definition}

\begin{definition}\cite{Harary}
	The {\bf line graph} $\mathcal{L}(G)$ of a graph $G$ is constructed by taking the edge set of $G$ as the vertex set of $\mathcal{L}(G)$. Two vertices in $\mathcal{L}(G)$ are connected by an edge if their corresponding edges in $G$ share a common vertex. The $k$-th {\bf iterated line graph} of $G$, denoted by $\mathcal{L}^k(G)$ for $k\in\{0, 1, 2,\ldots\}$, is defined as follows: $\mathcal{L}^0(G) = G$ and for $k \geq 1$, $\mathcal{L}^k(G) = \mathcal{L}(\mathcal{L}^{k-1}(G))$.
\end{definition}
Let $n_k$ and $m_k$ be the order and size of $\mathcal{L}^k(G)$ for $k \in \{0, 1, 2,\ldots\}$. It is noted that $n_k=m_{k-1}$ for $k \in \{1, 2,\ldots\}$. Let us denote the eigenvalues of $\mathcal{L}^k(G)$ and its complement $\overline{\mathcal{L}^k(G)} $ respectively  by $\lambda_{k(j)}$ and $\overline{\lambda}_{k(j)}$ for $k \in \{1, 2,\ldots\}$ and $1\le j\le n_k$. The complement of a line graph is also called jump graph \cite{Jump1}.

\begin{theorem}{\rm \cite{Cvetk_book}}\label{reg_compl}
	Suppose $G$ is an $r$-regular graph of order $n$ with its spectrum $Sp_A(G)=\{r, \lambda_2, \ldots, \lambda_n\}$, then its complement $\overline{G}$ is a $(n-r-1)$-regular graph with the spectrum $Sp_A(\overline{G})=\{n-r-1, -1-\lambda_n, \ldots, -1-\lambda_2\}$.
\end{theorem}
\begin{theorem}{\rm \cite{Sachs}}\label{line_eigv}
	Suppose $G$ is an $r$-regular graph of order $n$ and size $m$ with its spectrum $Sp_A(G)=\{r, \lambda_2, \ldots, \lambda_n\}$, then the line graph of $G$ is a $(2r-2)$-regular graph with the spectrum $Sp_A(\mathcal{L}(G))=\{2r-2, \lambda_2+r-2, \ldots, \lambda_n+r-2, -2^{m-n}\}$.
\end{theorem}
\noindent 
Let $n$ and $m$ represent the order and size of a graph $G$, respectively. The relationship between the eigenvalues of the line graph $\mathcal{L}(G)$ and the signless Laplacian $Q$ of $G$ is given by \cite{cvet_signless}: \begin{equation}\label{line_eig_snles_eqn} Sp_A\big(\mathcal{L}(G)\big) = \{-2^{m-n}\} \cup \big(Sp_Q(G)-2\big). \end{equation} The multiplicity of the eigenvalue $-2$ in $\mathcal{L}(G)$ is $m - n + 1$ if $G$ is bipartite, and $m - n$ if $G$ is not bipartite \cite{Cvetkovic_Sgen}.
\begin{theorem}{\rm\cite{Kinkar_Das}}\label{Kinkar}
	Let $G$ be a graph with order $n (> 2)$, $m$ edges and minimum degree $\delta$. If $\delta \ge \frac{n}{2} + 1$, then the line graph $\mathcal{L}(G)$ satisfies the property $\rho$ and has the energy of $4(m-n)$. Thus, all line graphs of such graphs with order $n$, size $m$ and minimum degree $\delta \ge \frac{n}{2} + 1$ are mutually equienergetic.
\end{theorem}

\begin{theorem}{\rm\cite{Hyper}}\label{hyper_en}
	Let $G$ be a graph of order $n(\ge 5)$ and size $m$. If $m\ge 2n$, the line graph $\mathcal{L}(G)$ is hyperenergetic.
\end{theorem}

\noindent The following is an elegant, though not sharp, relation between the smallest eigenvalue $\lambda_{\text{min}}$ of $A$ and the smallest eigenvalue $q_{\text{min}}$ of $Q$.

\begin{proposition}{\rm\cite{Desai}}\label{Desai_prop}
	If $G$ is a graph with minimum degree $\delta$ and maximum degree $\Delta$, then $$ q_{min}-\Delta \le \lambda_{min} \le q_{min}-\delta.$$
\end{proposition}

\begin{proposition}{\rm\cite{de_Lima}}\label{d_Lima}
	If $G$ is a spanning subgraph of a graph $G'$, then $q_{min}(G)\le q_ {min}(G')$.
\end{proposition}

\begin{theorem}{\rm\cite{HeChang}}\label{He_Chang}
	If $G$ is a graph with vertex $v$, then $q_{min}(G)-1\le q_{min}(G-v)$.
\end{theorem}

\noindent Let ${\bf j}$ denote all one's column vector, that is ${\bf j}=(1,1,\ldots, 1)^T$.
%If $\lambda$ is an eigenvalue of a graph of order $n$ then the set $\{Y\in {\mathbb{R}}^n : AY = \lambda Y\}$ is a subspace of $\mathbb{R}^n$, called the eigenspace of $\lambda$.

\begin{theorem}{\rm\cite{Hagos}}\label{compl-1-l}
	If $\lambda\in Sp_A(G)$, then $-1-\lambda\in Sp_A(\overline{G})$ if and only if ${\bf j}^TY=0$ for some eigenvector $Y$ of $G$ corresponding to the eigenvalue $\lambda$.
\end{theorem}

\begin{proposition}{\rm\cite{Cvetk_book}}\label{eig_space-2}
	In the line graph $\mathcal{L}(G)$ of a graph $G$ the eigenspace of the eigenvalue $-2$ is orthogonal to the vector ${\bf j}$.
\end{proposition}

\noindent The Weyl's  eigenvalue inequality \cite{Horn} $\lambda_j(M_1)+\lambda_k(M_2)\le \lambda_i(M_1+M_2)$, $j+k-n\ge i$ for sum of two Hermitian matrices $ M_1$ and $M_2$ of order $n$ gives the following useful eigenvalues inequality on a graph $G$ of order $n$ and its complement $\overline{G}$ \cite{Mojallal}:
\begin{equation}\label{eig_inq}
	\,\lambda_j+\overline{\lambda}_{n-j+2}\le -1 \text{ for } j\in \{2, 3, \ldots , n\}.
\end{equation}

\begin{proposition}{\rm\cite{GodsilCD}}\label{prop_minor}
	Let $M$ be any square matrix of order $n$ with the characteristic polynomial $\varphi(M;x)=\sum\limits_{r=0}^{n}(-1)^nm_rx^{n-r}$. Then $m_r$ is equal to the sum of the principal minors of $M$ of order $r$.  
\end{proposition}

\begin{definition}
	If $G$ is a graph of order $n$ with vertices $v_1, v_2, \ldots, v_n$, then the graph $G[H_1, H_2, \ldots, H_n]$ called {\bf generalized composition} or {\bf $H$-join} \cite{Schwenk, Cardoso_Spec} which is obtained from  the graphs $H_1, H_2,\ldots, H_n$ by joining every vertex of $H_i$ to every vertex of $H_j$ if $v_i$ is adjacent to $v_j$ in $G$.
\end{definition}

\noindent Let $\pi=(\pi_1,\pi_2,\ldots,\pi_p)$ be a partition of a vertex set of a graph $G$. The partition $\pi$ is called {\bf equitable partition} \cite{ChangAn,Schwenk} if for each $i,j=1,2,\ldots,p$ there exists a number $c_{ij}$  such that for every vertex $v\in\pi_i$  there are exactly $c_{ij}$ edges between $v$ and the vertices in $\pi_j$. If $\pi$ is an equitable partition then the associated $p\times p$ matrix  with rows and columns corresponding to the partite sets  $\pi_1,\pi_2,\ldots,\pi_p$ is called {\bf quotient matrix}. Let $A_{\pi}$  is the $p\times p$ matrix with $(i, j)$-th element $a^{\pi}_{ij}$ equal to $c_{ij}$ and the $p\times p$ diagonal matrix $D_{\pi}$ whose $i$-th diagonal element equal to $\sum\limits_{k=1}^{p}c_{ik}$. If $\pi$ is an equitable partition, we denote the quotient matrix for adjacency, Laplacian and signless Laplacian of a graph, respectively by $A_{\pi},  L_{\pi}$ and $Q_{\pi}$. The matrices $A_{\pi}, L_{\pi}$ and $Q_{\pi}$ are given by $A_{\pi}=[a^{\pi}_{ij}]$ \cite{Schwenk}, $L_{\pi}=[l^{\pi}_{ij}]=D_{\pi}-A_{\pi}$ \cite{Cardoso} and $Q_{\pi}=[q^{\pi}_{ij}]=D_{\pi}+A_{\pi}$ \cite{ChangAn,Wang_SL}. It is noted that these  need not be symmetric matrices.\\

Let $H_i$ be $r_i$-regular graphs of order $n_i$ for $i=1,2,\ldots,p$. In case of   $H$-join $G[H_1, H_2,\ldots, H_p]$, where $G$ is a graph of order $p$,  let us denote the quotient matrix for adjacency, Laplacian and signless Laplacian of a graph, respectively by $A^{H}_{\pi}=[a^{H}_{ij}],  L^{H}_{\pi}=[l^{H}_{ij}]$ and $Q^{H}_{\pi}=[q^{H}_{ij}]$. These matrices are defined as \cite{Cardoso_Spec,WuBaoFeng}

\begin{equation*}
	a^{H}_{ij} = \left\{ \begin{array}{ll} 
		c_{ij}        & \text{ if } {i=j}\\[3mm]
		\sqrt{n_in_j}        & \text{ if } {v_iv_j}\in E(G)\\[3mm]
		0 & \text{ otherwise,} \end{array} \right.
\end{equation*}

\begin{equation*}
	l^{H}_{ij} = \left\{ \begin{array}{ll} 
		l^{\pi}_{ij}        & \text{ if } {i=j}\\ [3mm]
		-\sqrt{n_in_j}        & \text{ if } {v_iv_j}\in E(G)\\ [3mm]
		0 & \text{ otherwise} \end{array} \right.
\end{equation*}
and
\begin{equation*}
	q^{H}_{ij}= \left\{ \begin{array}{ll} 
		q^{\pi}_{ij}        & \text{ if } {i=j}\\ [3mm]
		\sqrt{n_in_j}        & \text{ if } {v_iv_j}\in E(G)\\[3mm]
		0 & \text{ otherwise.} \end{array} \right.
\end{equation*}

\noindent It is noted that these  are symmetric matrices.\\

If $\pi=(\pi_1,\pi_2,\ldots,\pi_p)$ is an equitable partition of a graph $G$ with cardinality of $\pi_i$, $\lvert\pi_i\rvert = m_i$ for $i=1, \ldots, p$ then $\pi$ also equitable partition to its complement $\overline{G}$. The quotient matrix $\overline{A}_{\pi}$ of $\overline{G}$ is given by $\overline{A}_{\pi}={\bf J}_{\pi}-I-A_{\pi}$ \cite{Teranishi}, where ${\bf J}_{\pi}$ is the matrix of order $p$ whose $(i,j)$-th element equal to $m_j$ and $I$ is the identity matrix of order $p$.
\begin{proposition}{\rm\cite{Teranishi}}\label{Cospec_Prop}
	Let the graphs $G_1$ and $G_2$ are co-spectral with equitable partitions $\pi 1$ and $\pi 2$ respectively. If $A_{\pi 1} = A_{\pi 2}$, then the graphs $\overline{G_1}$ and $\overline{G_2}$ are co-spectral.
\end{proposition}
\begin{theorem}{\rm\cite{WuBaoFeng}}\label{SpecQ}
	Let $G$ be a graph of order $n$ and $H_i$ be $r_i$-regular graph of order $n_i$ for $i=1, 2, \ldots, n$. If ${\Gamma}=G[H_1, H_2,\ldots, H_n]$, then $$Sp_Q(\Gamma)=\left(\cup_{i=1}^n\Big((q^{\pi}_{ii}-2r_i)+(Sp_Q(H_i)\setminus\{2r_i\})\Big)\right)\cup \Big(Sp(Q^{H}_{\pi})\Big).$$
\end{theorem}

\noindent Suppose that $n\ge 2$ and $M=[m_{ij}]\in \mathbb{C}^{n\times n}$. The Ger\v{s}gorin discs $D_i, i=1,2,\ldots,n$ of the matrix $M$ are defined as the closed
circular regions $D_i=\{z\in\mathbb{C}:\lvert z-m_{ii}\rvert\le R_i\}$ in the complex plane, where
\begin{equation*}
	R_i=
	\sideset{}{}{\sum}_{\substack{ j=1\\ j\ne i }}^{n}\lvert m_{ij}\rvert
\end{equation*} is the radius of $D_i$.
\begin{theorem}[Ger\v{s}gorin]{\rm\cite{Horn}}\label{Gersgorin}
	Let $n\ge 2$ and $M\in \mathbb{C}^{n\times n}$. All eigenvalues of the matrix $M$ lie in the region $D=\cup_{i=1}^{n}D_i$, where $D_i, i=1,2,\ldots,n$ are the Ger\v{s}gorin discs of $M$.
\end{theorem}

\section{Spectra of Quotient Matrices}
\begin{theorem}\label{eqlquospec}
	Let $\Gamma = G[H_1, H_2,\ldots, H_n]$, where $H_i$ is an $r_i$-regular graph of order $n_i$ for $i=1, 2, \ldots, n$. Then the spectrum of the quotient matrices  $ A_{\pi},  L_{\pi}$ and $Q_{\pi}$ equal to the spectrum of the quotient matrices $A^{H}_{\pi},  L^{H}_{\pi}$ and $Q^{H}_{\pi}$ respectively.
\end{theorem}
\begin{proof}
	It is clear that $\pi = (V(H_1), V(H_2), \ldots, V(H_n))$ is an equitable partition of $\Gamma$. Let us first prove the result for the matrices $Q_{\pi}$ and $Q^{H}_{\pi}$ using Proposition \ref{prop_minor} and similar can be applied for the remaining matrices. The entries of the matrices $Q_{\pi}$ and $Q^{H}_{\pi}$ can be written as $q^{\pi}_{ij}=a_{ij}n_j$,  $q^{H}_{ij}=a_{ij}\sqrt{n_in_j}$ for $i\neq j$ and $q^{\pi}_{ii}=q^{H}_{ii}$, where $a_{ij}$ is the entry of the adjacency matrix of $G$. Let $S_n$ be the set of all permutations over the set $\{1, 2, \ldots, n\}$ and if $\sigma\in S_n$ denote its sign by $sgn(\sigma)$. Let $M_r^{'}$ be the principal minor of order $r$ which is obtained by deleting any $n-r$ rows and corresponding columns of the matrix $Q_{\pi}$ and the respective principal minor of the matrix $Q^{H}_{\pi}$ be $M_r^{''}$ for $1\leq r\leq n$. Then  $M_r^{'}=\sum\limits_{\sigma\in S_r}sgn(\sigma)\prod\limits_{k=1}^{r}q^{\pi}_{k\sigma(k)}=\sum\limits_{\sigma\in S_r}sgn(\sigma)\prod\limits_{k=1}^{r}a_{k\sigma(k)}n_{\sigma(k)}=\sum\limits_{\sigma\in S_r}sgn(\sigma)\prod\limits_{k=1}^{r}a_{k\sigma(k)}n_k$ and $M_r^{''}=\sum\limits_{\sigma\in S_r}sgn(\sigma)\prod\limits_{k=1}^{r}q^{H}_{k\sigma(k)}=\sum\limits_{\sigma\in S_r}sgn(\sigma)\prod\limits_{k=1}^{r}a_{k\sigma(k)}\sqrt{n_kn_{\sigma(k)}}=\sum\limits_{\sigma\in S_r}sgn(\sigma)\prod\limits_{k=1}^{r}a_{k\sigma(k)}n_k$ if $k\neq\sigma(k)$ for any $k=1,2,\ldots, r$. If $k=\sigma(k)$ for $k=1,2,\ldots, p$ and $1\leq  p\leq r$ then $M_r^{'}=M_r^{''}=\sum\limits_{\sigma\in S_r}sgn(\sigma)\prod\limits_{k=1}^{r-p}a_{k\sigma(k)}n_k(q^{\pi}_{kk})^p$. Which shows that $M_r^{'}=M_r^{''}$ to each $r$ for $1\le r\le n$. Hence, the sum of the principal minors of  order $r$ of the matrices $Q_{\pi}$ and $Q^{H}_{\pi}$ are equal. It shows that the matrices $Q_{\pi}$ and $Q^{H}_{\pi}$ have the same characteristic polynomials by Proposition \ref{prop_minor}. Thus the matrices $Q_{\pi}$ and $Q^{H}_{\pi}$ have the same spectrum.
\end{proof}
Now Theorem \ref{SpecQ}, with the help of above theorem can be stated in the following way:
\begin{proposition}\label{HSpecQeql}
	Let $G$ be a graph of order $n$ and $H_i$ be an $r_i$-regular graph of order $n_i$ for $i=1, 2, \ldots, n$. If $\Gamma=G[H_1, H_2$,\ldots, $H_n]$, then $$Sp_Q(\Gamma)=\left(\bigcup_{i=1}^n\Big(R_i+(Sp_Q(H_i)\setminus\{2r_i\})\Big)\right)\cup \Big(Sp(Q_{\pi})\Big),$$ where $R_i$ is the $i$-th row sum excluding the diagonal entry of $Q_{\pi}$ for  $i=1, 2, \ldots, n$.
\end{proposition}
\begin{proof}
	Proof follows by  Theorem \ref{eqlquospec} and an observation that $q^{\pi}_{ii} = 2r_i+R_i$ for   $i=1, 2, \ldots, n$.
\end{proof}

\section{Spectra and Energy of iterated line graphs with the property $\rho$}
\begin{theorem}\label{Iter_line-2}
	Let $G$ be a graph with order $n_0$ and size $m_0$, where each edge $e = uv$ in $G$ satisfies $d_u + d_v \ge 6$. Then, the graphs $\mathcal{L}^k(G)$ have the property $\rho$ for $k \ge 2$. Moreover, all the iterated line graphs $\mathcal{L}^k(G)$ of such a graph $G$ are mutually equienergetic, with energy $4(n_k - n_{k-1})$ for $k \ge 2$.
\end{theorem}
\begin{proof}
	If $G$ is a graph of order $n_0$ and size $m_0$ with an edge $e=uv$, then in the line graph $\mathcal{L}(G)$, the degree of the vertex corresponding to the edge $e$ is $d_u + d_v - 2$. Given that $d_u + d_v \ge 6$ for every edge $e = uv$ in $G$, it follows that $d_u + d_v - 2 \ge 4$. This implies that the minimum degree of each vertex in $\mathcal{L}(G)$ is at least four. It is well known that the least eigenvalue of the line graph $\mathcal{L}(G)$ is not smaller than $-2$. Therefore, by Proposition \ref{Desai_prop}, the least eigenvalue $\lambda_{\text{min}}$, the smallest signless Laplacian eigenvalue $q_{\text{min}}$, and the minimum degree $\delta$ of $\mathcal{L}(G)$ satisfy $q_{\text{min}} \ge \lambda_{\text{min}} + \delta \ge -2 + 4 = 2$. Now, by relation \eqref{line_eig_snles_eqn}, $\mathcal{L}(\mathcal{L}(G)) = \mathcal{L}^2(G)$ satisfies property $\rho$, with the multiplicity of $-2$ being equal to $m_1 - n_1$. It can be easily observed that the minimum degree $\delta$ in the line graphs $\mathcal{L}^k(G)$ increases as $k$ increases. This implies that $q_{\text{min}}$ of $\mathcal{L}^k(G)$ also increases with $k$ and is at least $2$. 	
	Hence, by relation \eqref{line_eig_snles_eqn}, the iterated line graphs $\mathcal{L}^k(G)$ satisfy property $\rho$ for each $k \ge 2$, with their energy equal to $4(m_{k-1} - n_{k-1}) = 4(n_k - n_{k-1})$.
\end{proof}
\noindent 
A tree graph is called a caterpillar if removing all its pendant vertices results in a path graph. The spectra and the energy of the line graphs of caterpillars are studied in \cite{OscarRojo}.
%\begin{center}
\begin{example}
	The graph in Figure 1 is the caterpillar $C(4,3,4)$. This graph satisfies the conditions of Theorem \ref{Iter_line-2}, so the graphs $\mathcal{L}^k\big(C(4,3,4)\big)$ satisfy the property $\rho$ for $k\ge 2$.	
	
	\begin{figure}[h]
		\centering
		\includegraphics[width=0.5\linewidth]{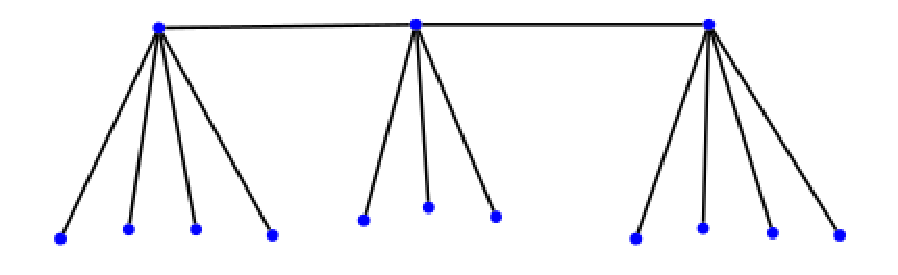}\\
		\caption{Caterpillar $C(4,3,4)$}
		\label{Fig1}
	\end{figure}
	
\end{example}

%\end{center}

\begin{corollary}\label{Cor_Iter_line-2}
	Let $G$ be a graph of order $n_0$ and size $m_0$ with a minimum degree $\delta \ge 3$. Then, the graphs $\mathcal{L}^k(G)$ satisfy property $\rho$ for $k \ge 2$. Moreover, all the iterated line graphs $\mathcal{L}^k(G)$ of such a graph $G$ are mutually equienergetic, with energy equal to $4(n_k - n_{k-1})$ for $k \ge 2$.
\end{corollary}

\begin{remark}
	In \cite{HSRamane}, Ramane et al. obtained the spectra and the energy of iterated line graphs of regular graphs with degree $r \ge 3$, thereby characterizing a large class of pairs of non-trivial equienergetic regular graphs. It is noted that all the results in their paper are special cases of Corollary \ref{Cor_Iter_line-2}.
\end{remark}
Now,  Theorem \ref{Iter_line-2} naturally leads us to consider when $\mathcal{L}(G)$ satisfies property $\rho$, given that $\mathcal{L}^k(G)$ satisfies property $\rho$ for $k \ge 2$ in a graph $G$ where $d_u + d_v \ge 6$ for each edge $e = uv$. In this context, we have the following results.

\begin{lemma}\label{leasteigQpi}
	Let $G$ be a connected graph of order $n(\ge 2)$ and $H_i$ be $r_i$-regular graph of order $n_i$, where $r_i\ge 1$, for $i=1, 2, \ldots, n$. Then the least eigenvalue of the quotient matrix $Q_{\pi}$ of the graph $G[H_1, H_2,\ldots, H_n]$ is at least  $2$. 
\end{lemma}

\begin{proof}
	The signless Laplacian matrix $Q$ of a graph $G$ is positive semidefinite, which possesses real eigenvalues. Consequently, by Proposition \ref{HSpecQeql}, the eigenvalues of the quotient matrix $Q_{\pi}$ are also real. In the $i$-th row of $Q_{\pi}$, the diagonal entry $q^{\pi}_{ii}$ satisfies $q^{\pi}_{ii} = 2r_i + R_i$, where $R_i$ represents the sum of the $i$-th row of $Q_{\pi}$ excluding $q^{\pi}_{ii}$. By applying Ger\v{s}gorin’s  Theorem \ref{Gersgorin}, all the eigenvalues of $Q_{\pi}$ are contained within the union of the closed intervals $[2r_i, 2(r_i + R_i)]$ for $i = 1, 2, \dots, n$. Now, the result follows from the condition that $r_i \geq 1$.
	
\end{proof} 

\begin{theorem}\label{lineofjoin}
	Let $G$ be a connected graph of order $n(\ge 2)$ and $H_i$ be an $r_i$-regular graph of order $p_i$, where $r_i\ge 1$, for $i=1, 2, \ldots, n$. If ${\Gamma}=G[H_1, H_2, \ldots, H_n]$ is a graph of order $n_0$ and size $m_0$  then the graphs $\mathcal{L}^k(\Gamma)$ satisfy the property $\rho$ for $k\ge 1$. Furthermore, the line graphs of graphs $G'$ of order $n_k$ and size $m'_{k}$ for which $\mathcal{L}^k(\Gamma)$ is a spanning subgraph are mutually equienergetic with energy $4(m'_k-n_k)$ for $k\ge 0$.
\end{theorem}

\begin{proof}
	Since $ G$ is a connected graph of order $ n \geq 2$, for every vertex $ v_i$ in $ G$, there exists at least one adjacent vertex $ v_j$. This implies that each vertex of the graph $ H_i$ is adjacent to every vertex of $ H_j$ for at least one $ j$ in $ \Gamma = G[H_1, H_2, \ldots, H_n]$. Moreover, since $ H_i$ are $ r_i$-regular graphs of order $ p_i$, with $ r_i \geq 1$, it follows that $ p_i \geq 2$. Thus the quotient matrix $ Q_{\pi}$ corresponding to the equitable partition $ \pi = (V(H_1), V(H_2), \ldots, V(H_n))$ has a non-diagonal entry $ q^{\pi}_{ij} \geq 2$ in its $ i$-th row for at least one $ j$.	
	Since the signless Laplacian $ Q$ of any graph is positive semidefinite, it has non-negative eigenvalues. By applying Proposition \ref{HSpecQeql} and Lemma \ref{leasteigQpi}, we conclude that every signless Laplacian eigenvalue of $ \Gamma$ is at least $2$. From the relation \eqref{line_eig_snles_eqn}, it follows that the line graph $ \mathcal{L}(\Gamma)$ satisfies the property $ \rho$.
	The minimum order graph $ \Gamma$ with the smallest possible degrees is $ \Gamma = K_2[K_2, K_2]$, which shows that the minimum degree $ \delta$ of $ \Gamma$ is at least $3$. Thus, by Corollary \ref{Cor_Iter_line-2} all the iterated line graphs $ \mathcal{L}^k(\Gamma)$ for $ k \geq 1$ satisfy the property $ \rho$. Since $ \Gamma$ is a spanning subgraph of $ G'$, Proposition \ref{d_Lima} implies that $ G'$ also has $ q_{\min} \geq 2$, which further implies that the graph $ \mathcal{L}(G')$ satisfies the property $ \rho$ with its energy equal to $ 4(m'_k - n_k)$ for $ k \geq 0$ by relation \eqref{line_eig_snles_eqn}. This completes the proof.
\end{proof}

\begin{remark}\label{Join_remark}
	Let $G$ be a connected graph of order $n\ge 2$ and $H_i$ be an $r_i$-regular graph of order $p_i$  with $r_i\ge 1$, for $i=1, 2, \ldots, n$. Let ${\Gamma}=G[H_1, \ldots, H_j, \ldots, H_n]$ for  $1\le j\le i$ and $H_{s_1}$, $H_{s_2}$ be two $r$-regular graphs of same order with $r\ge 1$. If ${\Gamma_1}=G[H_1, \ldots, H_{s_1}, \ldots, H_n]$ and ${\Gamma_2}=G[H_1, \ldots, H_{s_2}, \ldots, H_n]$, then the graphs  $\mathcal{L}^k(\Gamma_1)$ and $\mathcal{L}^k(\Gamma_2)$ are equienergetic for $k\ge 1$. Further, if $H_{s_1}$ and $H_{s_2}$ are non co-spectral (co-spectral) graphs then we get non co-spectral (co-spectral) graphs $\mathcal{L}^k(\Gamma_1)$ and $\mathcal{L}^k(\Gamma_2)$ respectively, for $k\ge 0$ as they have same quotient matrices. 
\end{remark}

\begin{remark}
	In Theorem \ref{Kinkar}, Das et al. characterized a large class of equienergetic line graphs of graphs of order $ n$ under the condition that the minimum degree $ \delta \geq \frac{n}{2} + 1$. However, one can also construct equienergetic line graphs of graphs of order $ n$ with the condition $ \delta \leq \frac{n}{2} + 1$ by using Theorem \ref{lineofjoin}. For example, if $ \Gamma = K_2[K_2, C_n]$ with $ n \geq 6$, it is always possible to construct non-isomorphic equienergetic line graphs of $ \Gamma$ with energy equal to $4$ times (size of $\Gamma-$ order of  $\Gamma)$ for $ \delta = 4$ by adding edges between non-adjacent vertices of $ \Gamma$.
	
\end{remark}
\begin{theorem}\label{vertex_del}
	Let $G$ be a connected graph of order $n(\ge 2)$ and $H_i$ be an $r_i$-regular graph of order $p_i$, where $r_i\ge 2$, for $i=1, 2, \ldots, n$. If ${\Gamma}=G[H_1, H_2,\ldots, H_n]$, then the graphs $\mathcal{L}^k(\Gamma-v)$ satisfy the property $\rho$ for $k\ge 1$. Moreover, if $n_i\ge min\{2r_i: r_i\ge s\ge2\}$ then the graphs $\mathcal{L}^k(\Gamma-\{v_1, v_2, \ldots, v_{2(s-1)}\})$ satisfy the property $\rho$ for $k\ge 1$.
\end{theorem}
\begin{proof}
	All the eigenvalues of $ Q_{\pi} $ of $ \Gamma $ belong to the union of the closed intervals $ [2r_i, 2(r_i+R_i)] $ for $ i=1, 2, \ldots, n $ by Lemma \ref{leasteigQpi}, which shows that each eigenvalue of $ Q_{\pi} $ is at least $ 4 $ since $ r_i \ge 2 $. Each row of the matrix $ Q_{\pi} $ has at least one non-diagonal entry that is at least $ 3 $ as $ r_i \ge 2 $ and $ G $ is connected with order $ n \ge 2 $. By Proposition \ref{HSpecQeql}, the $ q_{\min} $ of $ \Gamma $ is at least $ 3 $. By Theorem \ref{He_Chang}, the $ q_{\min} $ of $ \Gamma-v $ is at least $ 2 $. Hence $ \mathcal{L}(\Gamma-v) $ satisfies the property $ \rho $ by relation \eqref{line_eig_snles_eqn}. 
	If $ n_i \ge \min\{2r_i: r_i \ge s \ge 2\} $, then the $ q_{\min} $ of $ Q_{\pi} $ is at least $ 2s $, and each row of $ Q_{\pi} $ has at least one non-diagonal entry at least $ 2s $ with $ s\ge 2 $. By Proposition \ref{HSpecQeql}, the $ q_{\min} $ of $ \Gamma $ is at least $ 2s $. It is easy to observe by Theorem \ref{He_Chang} that 
	$
	q_{\min}(\Gamma) - 2(s-1) \le q_{\min}\big(\Gamma-\{v_1, v_2, \ldots, v_{2(s-1)}\}\big).
	$
	Hence $ \mathcal{L}\big(\Gamma-\{v_1, v_2, \ldots, v_{2(s-1)}\}\big) $ satisfies the property $ \rho $ by relation \eqref{line_eig_snles_eqn}. 
	The graphs $ \Gamma-v $ for $ r_i \ge 2 $ and $ \Gamma-\{v_1, v_2, \ldots, v_{2(s-1)}\} $ for $ n_i \ge \min\{2r_i: r_i \ge s \ge 2\} $ both are graphs with minimum degree $ \delta \ge 3 $. Thus by Corollary \ref{Cor_Iter_line-2}, all the iterated line graphs $ \mathcal{L}^k(\Gamma-v) $ and $ \mathcal{L}^k(\Gamma-\{v_1, v_2, \ldots, v_{2(s-1)}\}) $ for $ k \ge 1 $ satisfy the property $ \rho $.
	
\end{proof}

\noindent The following is an interesting result due to the minimum number of edges in join of two connected graphs. 

\begin{corollary}\label{path_join_ite}
	If $m, n\ge 3$ then the graphs $\mathcal{L}^k\big(K_2[P_n, P_m]\big)$ satisfy the property $\rho$ for $k\ge 1$.
\end{corollary}
\begin{proof}
	Let $ G = K_2 $, $ H_1 = C_{n+1} $ and $ H_2 = C_{m+1} $ for $ m, n \ge 3 $ in Theorem \ref{vertex_del}. Let $ v_1 \in H_1 $ and $ v_2 \in H_2 $. By deleting the vertices $ v_1 $ and $ v_2 $ along with the edges incident to them in $ H_1 $ and $ H_2 $ respectively, we obtain $ \Gamma - \{v_1, v_2\} = K_2[P_n, P_m] $. Therefore, the graphs $ \mathcal{L}^k(K_2[P_n, P_m]) $ satisfy the property $ \rho $ for $ k \ge 1 $.
\end{proof}

\begin{definition}\cite{Gutman_Kn}
	Let $ v $ be a vertex of the complete graph $ K_n $ where $ n \ge 3 $, and let $ e_i $ for $ i = 1, \ldots, p $, $ 1 \le p \le n-1 $ be its distinct edges all incident to $ v $. The graph $ Ka_n(p) $ is obtained by deleting the edges $ e_i $ for $ i = 1, \ldots, p $ from $ K_n $. Note that $ Ka_n(0) = K_n $.
	
\end{definition}

\noindent With this notation we have the following result.
\begin{theorem}\label{ver_del_Kn}
	If $n\ge 6$ then the graphs $\mathcal{L}^k\big(Ka_n(p)\big)$, $1\le p\le n-4$   satisfy the property $\rho$ for $k\ge 1$.
\end{theorem}
\begin{proof}
	All the graphs of order up to 5 whose line graphs satisfy the property $ \rho $ are $ C_4 $, $ K_4 $, $ K_{3,2} $ and $ K_5 $ according to Theorem \ref{min_order_thm}. It is noted that none of these graphs are of the type $ Ka_n(p) $ for $ p \ge 1 $. If $ n \ge 6 $, the graph $ Ka_n(n-4) $ can be expressed as $ P_3[K_1, K_3, K_{n-4}] $. The quotient matrix $ Q_\pi $ of $ Ka_n(n-4) $ is given by
	\[
	Q_\pi = \begin{bmatrix}
		3   & 3   &   0 \\
		1   & n+1   &  n-4 \\
		0   & 3   &   2n-7 
	\end{bmatrix}
	\]
	with its spectrum $$ Sp(Q_\pi) = \left\{n-\frac{1}{2}+\frac{1}{2}\sqrt{4n(n-7)+73}, n-2, n-\frac{1}{2}-\frac{1}{2}\sqrt{4n(n-7)+73}\right\}. $$ The signless Laplacian spectrum of $ Ka_n(n-4) $ is 
	\[
	Sp_Q\big(Ka_n(n-4)\big) = \{ (n-2)^2,  (n-3)^{n-5}\} \cup Sp(Q_\pi).
	\]
	It is clear that all the signless Laplacian eigenvalues of $ Ka_n(n-4) $ are greater than or equal to $ 2 $, except for the concern about the third eigenvalue of $ Sp(Q_\pi) $ when $ n \ge 6 $. However, this eigenvalue $ n-\frac{1}{2}-\frac{1}{2}\sqrt{4n(n-7)+73} \ge 2 $ if $ n \ge 5 $. Thus, $ \mathcal{L}\big(Ka_n(p)\big) $ for $ 1 \le p \le n-4 $ satisfies the property $ \rho $ by using Proposition \ref{d_Lima} and the relation \eqref{line_eig_snles_eqn}. 
	It is also easy to observe that the minimum degree of $ Ka_n(p) $ for $ 1 \le p \le n-4 $ is at least $ 3 $. Hence by Corollary \ref{Cor_Iter_line-2}, all the iterated line graphs $ \mathcal{L}^k(Ka_n(p)) $ for $ k \ge 1 $ and $ 1 \le p \le n-4 $ satisfy the property $ \rho $.
	
\end{proof} 
There are certain classes of graphs with a least eigenvalue of $-2$, such as exceptional graphs and generalized line graphs \cite{Cvetkovic_Sgen}. If the minimum degree $ \delta \ge 4 $ in these classes of graphs, we have the following simple result.

\begin{theorem}\label{min_deg4}
	Let $G$ be a graph with least eigenvalue $-2$ and the minimum degree $\delta\ge 4$. Then the  iterated line graphs $\mathcal{L}^k(G)$ satisfy the property $\rho$ for  $k\ge 1$.
\end{theorem}

\begin{proof}
	The least eigenvalues $\lambda_{min}, q_{min}$ and the minimum degree $\delta$ of $G$ by Proposition \ref{Desai_prop} satisfy $q_{min}\ge \lambda_{min}+\delta\ge -2+4=2$. Therefore, by using the relation \eqref{line_eig_snles_eqn} $\mathcal{L}(G)$ satisfies the property $\rho$.  Now, by using 
	Corollary \ref{Cor_Iter_line-2} all the iterated line graphs $\mathcal{L}^k(G)$  satisfy the property $\rho$ for $k\ge 1$.
\end{proof}

\noindent Theorem \ref{Kinkar} of Das et al. can be extended to the iterated line graphs.

\begin{theorem}\label{Kinkar_ite_lin}
	Let $G$ be a graph of order $n_0( > 2)$ and size $m_0$ with the minimum degree $\delta\ge \frac{n_0}{2} + 1$. Then the graphs $\mathcal{L}^k(G)$  satisfy the property $\rho$ for $k\ge 1$.
\end{theorem}

\begin{proof}
	If $G$ is a graph of order $n_0 > 2$ and size $m_0$ with the minimum degree $\delta\ge \frac{n_0}{2} + 1$, then by Theorem \ref{Kinkar}, $\mathcal{L}(G)$ satisfies the property $\rho$. The existence of a graph $G$ with the minimum degree $\delta\ge \frac{n_0}{2} + 1$ implies $n_0\ge 4$, that is,  the minimum degree of $G$ is at least $3$. Therefore, by applying 
	Corollary \ref{Cor_Iter_line-2}, all the iterated line graphs $\mathcal{L}^k(G)$  satisfy the property $\rho$ for $k\ge 1$.
\end{proof}

\noindent The following inequality was given by Leonardo de Lima et al. in \cite{Leonardo} for Tur\'{a}n graph.
\begin{equation*}
	(r-2)\Big\lfloor \frac{n}{r}\Big\rfloor<q_{min}\big(T_r(n)\big)\le \Big(1-\frac{1}{r}\Big)n.
\end{equation*}

\begin{note}
	The above inequality is not valid when $r=3$, $n=6, 7, 8$ and   $r=4$, $n=5$ as seen from the following spectral values:  $Sp_Q\big(T_3(6)\big)=\{8, 4^3, 2^2\}$,  $Sp_Q\big(T_3(7)\big)=\{9.2745, 5^{2}, 4^{2}, 3, 1.7251\}$,  $Sp_Q\big(T_3(8)\big)=\{10.6056, 6, 5^4, 3.3944, 2\}$ and $Sp_Q\big(T_4(5)\big)=\{7.3723, 3^3, 1.6277\}$. These values demonstrate that $q_{min}\big(T_3(6)\big)=2, q_{min}\big(T_3(7)\big)=1.7251$, $q_{min}\big(T_3(8)\big)=2 $  and $q_{min}\big(T_4(5)\big)=1.6277$, but the inequality gives strict lower bound $2$ for $q_{min}\big(T_r(n)\big)$. However, with the help of this  inequality we have the following result.
\end{note}

\begin{proposition}\label{Turan_itera}
	If $r = 3$ and $n\ge 6$; $n\ne 7$ or $r=4$ and $n\ne 5$  or $r\ge 5$, then the graphs $\mathcal{L}^k\big(T_r(n)\big)$  satisfy the property $\rho$ for $k\ge 1$.
\end{proposition}

\begin{proof}
	The condition on $r$ and $n$ in the hypothesis together with the above inequality guarantees that $q_{\text{min}}(T_r(n))$ is at least $2$. As a result, $\mathcal{L}(T_r(n))$ satisfies the property $\rho$ by applying the relation \eqref{line_eig_snles_eqn}. Moreover, it is evident that the minimum degree of $T_r(n)$ is at least $3$. Consequently, by Corollary \ref{Cor_Iter_line-2} all iterated line graphs $\mathcal{L}^k(G)$ satisfy the property $\rho$ for $k \geq 1$.
	
\end{proof}

\subsection{\bf Iterated regular line graphs with property $\rho$}
Most of the results discussed so far involve non-regular iterated line graphs that satisfy the property $\rho$. We obtain iterated regular line graphs $\mathcal{L}^k(G)$ that satisfy the property $\rho$ from Theorem \ref{Iter_line-2} for $k \geq 2$ and from Theorem \ref{lineofjoin}, Theorem \ref{min_deg4} and Theorem \ref{Kinkar_ite_lin} for $k \geq 1$. In Proposition, \ref{Turan_itera} if $r$ divides $n$, we also get iterated regular line graphs $\mathcal{L}^k(G)$ that satisfy the property $\rho$ for $k \geq 1$. Here, we present additional iterated regular line graphs $\mathcal{L}^k(G)$ for $k \geq 1$ by taking regular graphs $G$.

\begin{theorem}\label{reg1_ite_lin}
	If $G$ is an $r$-regular graph of order $n$ with $3\le r\le \frac{n-1}{3}$, then  the graphs $\mathcal{L}^k(\overline{G})$ satisfy the property $\rho$ for $k\ge 1$.
\end{theorem}
\begin{proof}
	Let $Sp_A(G)=\{r, \lambda_2^{m_2}, \ldots, \lambda_t^{m_t}\}$ such that $1+\sum\limits_{i=2}^{t}m_i=n$. Then, by Theorem \ref{reg_compl},  $\overline{G}$ is also a regular graph with  $Sp_A(\overline{G})=\{n-r-1, (-1-\lambda_t)^{m_t}, \ldots, (-1-\lambda_2)^{m_2}\}$ and by Theorem \ref{line_eigv}, $Sp_A(\mathcal{L}(\overline{G}))=\{2(n-1)-2r-2, (n-r-\lambda_t-4)^{m_t},  \ldots, (n-r-\lambda_2-4)^{m_2}, -2^{n(n-r-3)/2}\}$. We shall prove that all eigenvalues of $\mathcal{L}(\overline{G})$, except $-2$ are non negative.  Since $G$ is regular, $\mathcal{L}(\overline{G})$ is also regular with degree $2(n-1)-2r-2$, which implies  $2(n-1)-2r-2>0$. The condition $r\le \frac{n-1}{3}$ gives  $n\ge 3r+1$ which implies $n-r-\lambda_i-4\ge 2r-\lambda_i-3\ge 0$ as $r\ge 3$ to each $i=2, \ldots, t$. Hence $\mathcal{L}(\overline{G})$ satisfies the property $\rho$. Additionally, $n\ge 3r+1$ implies $n-r-1\ge 2r\ge 6$ which shows that $\overline{G}$ is a regular graph with a minimum degree $\ge 6$. This implies that the graphs $\mathcal{L}^k(\overline{G})$ satisfy the property $\rho$ to each $k\ge 1$ by Corollary \ref{Cor_Iter_line-2}.
\end{proof}

\begin{theorem}\label{reg2_ite_lin}
	If $G$ is an $r$-regular graph of order $n\ge 8$ and $r\ge 1$, then the graphs  $\mathcal{L}^k\big(\overline{\mathcal{L}(G)}\big)$, $k\ge 1$ satisfy the property $\rho$.
\end{theorem}
\begin{proof}
	Let $Sp_A(G)=\{r, \lambda_2^{m_2}, \ldots, \lambda_t^{m_t}\}$ such that $1+\sum\limits_{i=2}^{t}m_i=n$. Then, by Theorem \ref{line_eigv}, $Sp_A(\mathcal{L}(G))=\{2r-2, (\lambda_2+r-2)^{m_2}, \ldots, (\lambda_t+r-2)^{m_t}, -2^{n(r-2)/2}\}$. Since $\mathcal{L}(G)$ is regular by Theorem \ref{reg_compl}, $\overline{\mathcal{L}(G)}$ also a regular graph with $Sp_A\big(\overline{\mathcal{L}(G)}\big)=\{nr/2-2r+1, 1^{n(r-2)/2},  (1-r-\lambda_t)^{m_t},  \ldots, (1-r-\lambda_2)^{m_2}\}$. Again, by Theorem \ref{line_eigv}, $Sp_A\big(\mathcal{L}\big(\overline{\mathcal{L}(G)}\big)\big) = \{r(n-4), ((n-4)r/2)^{n(r-2)/2}, ((n-6)r/2-\lambda_t)^{m_t}, \ldots, ((n-6)r/2-\lambda_2)^{m_2},  -2^{nr(nr-4r-2)/8}\}$. It is easy to see that all the eigenvalues of $\mathcal{L}\big(\overline{\mathcal{L}(G)}\big)$ are non-negative except $-2$ for $n\ge 8$. Hence,  $\mathcal{L}\big(\overline{\mathcal{L}(G)}\big)$ satisfies the property $\rho$. It is noted that the degree of $\overline{\mathcal{L}(G)}$ is $nr/2-2r+1$, which is at least $3$ for $n\ge 8$ and $r\ge 1$. This implies the graphs $\mathcal{L}^k\big(\overline{\mathcal{L}(G)}\big)$ satisfy the property $\rho$ for $k\ge 1$ by Corollary \ref{Cor_Iter_line-2}.
\end{proof}

\begin{remark}
	One can easily construct equienergetic graphs, similar to those in Theorem \ref{lineofjoin} for the results in Theorems \ref{vertex_del} to \ref{reg2_ite_lin}.
\end{remark}

\begin{theorem}\label{hyper_en_itLine}
	Let $G$ be a graph with  $d_u+d_v\ge 6$ to each edge $e=uv$ in  $G$. Then the graphs $\mathcal{L}^k(G)$ are hyperenergetic for  $k\ge 2$.
\end{theorem}
\begin{proof}
	Since $G$ is a graph where $d_u + d_v \geq 6$ for each edge $e = uv$, $\mathcal{L}(G)$ is a graph with minimum degree $\delta \geq 4$. The number of edges in $\mathcal{L}(G)$ is $m_1 = \frac{1}{2} \sum\limits_{i=1}^{m_0} d_i \geq \frac{1}{2}(4m_0) = 2m_0 = 2n_1$, which implies that the graph $\mathcal{L}^2(G)$ is hyperenergetic by Theorem \ref{hyper_en}. Note that the minimum degree increases in the line graphs $\mathcal{L}^k(G)$ as $k$ increases for $k \geq 2$ and it is at least 4. Therefore, $m_k \geq 2n_k$ and by Theorem \ref{hyper_en}, all iterated line graphs $\mathcal{L}^k(G)$ are hyperenergetic for $k \geq 2$.
	
\end{proof}

\section{Spectra and Energy of complement of iterated line graphs with the property $\rho$}

\begin{lemma}\label{compl_lem}
	Let $\mathcal{L}(G)$ be the line graph of a graph $G$ with order $n_0$ and size $m_0$. If $\mathcal{L}(G)$ has a non-negative eigenvalue $\lambda_{1(j)}$, then its complement $\overline{\mathcal{L}(G)}$ has a negative eigenvalue $\overline{\lambda}_{1(m_0-j+2)}$ for $j\in \{2, 3, \ldots , m_0\}$. In a particular, if $\mathcal{L}(G)$ has eigenvalue $-2$, then its complement $\overline{\mathcal{L}(G)}$ has  eigenvalue $1$.
\end{lemma}

\begin{proof}
	If $G$ is a graph of order $n_0$ and size $m_0$, then its line graph $\mathcal{L}(G)$ has order $m_0$. If $\mathcal{L}(G)$ has  a  non-negative eigenvalue $\lambda_{1(j)}$, then its complement $\overline{\mathcal{L}(G)}$ has eigenvalue an $\overline{\lambda}_{1(m_0-j+2)}\le -1-\lambda_{1(j)}$ by  inequality \ref{eig_inq}, which is negative for $j\in \{2, 3, \ldots , m_0\}$. If $-2$ is the eigenvalue of $\mathcal{L}(G)$, then by Proposition \ref{eig_space-2}, its eigenspace is orthogonal to all one's vector ${\bf j}$. Now by using Theorem \ref{compl-1-l}, the eigenvalue  $-1-(-2)=1$ is the eigenvalue of $\overline{\mathcal{L}(G)}$. This completes the proof.
\end{proof}
The papers \cite{Mojallal, Nikiforov, ramathesis, Ramane_Bp, Ramane-2023} explore exact relations between a regular graph $G$ and its complement. In the following, we extend these results to non-regular iterated line graphs with property $\rho$.
\begin{theorem}\label{comple_spec_enrg}
	Let $G$ be a graph of order $n_0$ and size $m_0$. If the graphs ${\mathcal{L}^k(G)}$  satisfy  the property $\rho$ with $-2$  multiplicity  $m_{k-1}-n_{k-1}$ for  $k\ge 1$, then the graphs $\overline{\mathcal{L}^k(G)}$ for  $k\ge 1$ have exactly two positive eigenvalues: the spectral radius and $1$ with multiplicity $m_{k-1}-n_{k-1}$. Furthermore \begin{equation}\label{rel_line_cpml}
		\mathcal{E}(\overline{\mathcal{L}^k(G)})=2\overline{\lambda}_{k(1)}+\frac{\mathcal{E}(\mathcal{L}^k(G))}{2}.\end{equation}
\end{theorem}
\begin{proof}
	It is given that the iterated line graphs $\mathcal{L}^k(G)$ of $G$ for $k \ge 1$ have all negative eigenvalues equal to $-2$ with multiplicity $m_{k-1} - n_{k-1} = n_k - n_{k-1}$. This implies that the remaining $n_{k-1}$ eigenvalues of $\mathcal{L}^k(G)$ are non-negative. By Lemma \ref{compl_lem}, the complement of the iterated line graphs $\overline{\mathcal{L}^k(G)}$ has negative eigenvalues $\overline{\lambda}_{k(n_k-j+2)}$ for $j \in \{2, 3, \ldots, n_{k-1}\}$ and positive eigenvalues $\overline{\lambda}_{k(n_k-j+2)} = 1$ for $j \in \{n_{k-1}+1, \ldots, n_k\}$.
	The only remaining eigenvalue of $\overline{\mathcal{L}^k(G)}$ is the spectral radius $\overline{\lambda}_{k(1)}$, which must be greater than or equal to $1$. If $\overline{\mathcal{L}^k(G)}$ is connected, then $\overline{\lambda}_{k(1)} > 1$. Thus, $\overline{\mathcal{L}^k(G)}$ has exactly two positive eigenvalues: $\overline{\lambda}_{k(1)}$ and $1$, with the latter having multiplicity $n_k - n_{k-1}$ for $k \ge 1$. Therefore, the energy of $\overline{\mathcal{L}^k(G)}$ is $\mathcal{E}(\overline{\mathcal{L}^k(G)}) = 2(\overline{\lambda}_{k(1)} + n_k - n_{k-1})$.
	Since $\mathcal{E}({\mathcal{L}^k(G)}) = 4(n_k - n_{k-1})$ as $-2$ is the only negative eigenvalue of ${\mathcal{L}^k(G)}$ with multiplicity $n_k - n_{k-1}$, we obtain the required energy relation between $\mathcal{E}(\overline{\mathcal{L}^k(G)})$ and $\mathcal{E}({\mathcal{L}^k(G)})$ for $k \ge 1$.
\end{proof}

\begin{corollary}\label{nec_suf_copl_equ}
	Let $G$ be a graph of order $n_0$ and size $m_0$. If the graphs ${\mathcal{L}^k(G)}$  satisfy the property $\rho$ with $-2$  multiplicity  $m_{k-1}-n_{k-1}$ for  $k\ge 1$, then the graphs ${\mathcal{L}^k(G)}$ and $\overline{\mathcal{L}^k(G)}$ for $k\ge 1$ are
	equienergetic if and only if $n^-(\mathcal{L}^k(G))=\overline{\lambda}_{k(1)}$.
\end{corollary}

\begin{proof}
	The energy relation \eqref{rel_line_cpml} between $\mathcal{E}(\overline{\mathcal{L}^k(G)})$ and $\mathcal{E}(\mathcal{L}^k(G))$ by Theorem \ref{comple_spec_enrg} can be expressed as $2\mathcal{E}(\overline{\mathcal{L}^k(G)}) - \mathcal{E}(\mathcal{L}^k(G)) = 4\overline{\lambda}_{k(1)}$, or equivalently, $\mathcal{E}(\overline{\mathcal{L}^k(G)}) - \mathcal{E}(\mathcal{L}^k(G)) = 4\overline{\lambda}_{k(1)} - \mathcal{E}(\overline{\mathcal{L}^k(G)})$. Now, the graphs $\mathcal{L}^k(G)$ and $\overline{\mathcal{L}^k(G)}$ for $k \ge 1$ are equienergetic if and only if $4\overline{\lambda}_{k(1)} - \mathcal{E}(\overline{\mathcal{L}^k(G)}) = 0$ or $4\overline{\lambda}_{k(1)} = \mathcal{E}(\overline{\mathcal{L}^k(G)})$. Since $\mathcal{E}(\overline{\mathcal{L}^k(G)}) = 2(\overline{\lambda}_{k(1)} + n_k - n_{k-1})$, we obtain $\overline{\lambda}{k(1)} = n_k - n_{k-1}$. But, it is given that $n^-(\mathcal{L}^k(G)) = n_k - n_{k-1}$, which completes the proof.
\end{proof}

%\begin{remark}
%	Recently Mojallal and Hansen in \cite{Mojallal}, the necessary and sufficient condition for a regular graph $G$ to be equienergetic with its complement $\overline{G}$ is given. In corollary \ref{nec_suf_copl_equ}, we extended the same to the non-regular iterated line graphs $\mathcal{L}^k(G)$.
%\end{remark}

\begin{corollary}\label{compl_equi}
	Let $G$ be a graph of order $n_0$ and size $m_0$. If the graphs ${\mathcal{L}^k(G)}$ satisfy the property $\rho$ with $-2$ having multiplicity $m_{k-1} - n_{k-1}$ for $k \ge 1$, then the complements of the iterated line graphs $\overline{\mathcal{L}^k(G)}$ are mutually equienergetic for $k \ge 1$ if and only if they have the same spectral radius.	
\end{corollary}

\begin{proof}
	If $G$ is a graph of order $n_0$ and size $m_0$, then the graphs ${\mathcal{L}^k(G)}$ have order $n_k$ and size $m_k$. If the graphs ${\mathcal{L}^k(G)}$ satisfy the property $\rho$ with $-2$ having multiplicity $m_{k-1} - n_{k-1}$ for $k \ge 1$, then all the iterated line graphs $\mathcal{L}^k(G)$ of such graphs $G$ with the same order $n_0$ and the same size $m_0$ are mutually equienergetic with energy $4(m_{k-1} - n_{k-1})$. Using this fact in the energy relation \eqref{rel_line_cpml} between $\mathcal{E}(\overline{\mathcal{L}^k(G)})$ and $\mathcal{E}(\mathcal{L}^k(G))$ by Theorem \ref{comple_spec_enrg} completes the proof.
\end{proof}

\begin{example}
	The graphs $\overline{\mathcal{L}^k(\Gamma_1)}$ and $\overline{\mathcal{L}^k(\Gamma_2)}$ in Remark \ref{Join_remark} are equienergetic for $k \ge 1$ as they have the same quotient matrices and because the spectral radius of a quotient matrix coincides with the spectral radius of the corresponding graph. Moreover, in Remark \ref{Join_remark} if $H_{s_1}$ and $H_{s_2}$ are non co-spectral (co-spectral) graphs, then we obtain non co-spectral (co-spectral) graphs $\overline{\mathcal{L}^k(\Gamma_1)}$ and $\overline{\mathcal{L}^k(\Gamma_2)}$ respectively for $k \ge 0$ by using Proposition \ref{Cospec_Prop}.
\end{example} 

\begin{remark}
	If $k \ge 1$, then the results in Theorem \ref{comple_spec_enrg}, Corollary \ref{nec_suf_copl_equ} and Corollary \ref{compl_equi} hold true for the iterated line graphs ${\mathcal{L}^k(G)}$ in Theorems \ref{lineofjoin} to \ref{reg2_ite_lin}. Similarly, if $k \ge 2$ these results hold true for the iterated line graphs ${\mathcal{L}^k(G)}$ in Theorem \ref{Iter_line-2} and Corollary \ref{Cor_Iter_line-2}.	
\end{remark}

\begin{remark}
	In \cite{ramane_copl_equi}, Ramane et al. obtained equienergetic regular graphs using the complement of iterated regular line graphs $\overline{\mathcal{L}^k(G)}$ for $k\ge 2$ by taking regular graphs $G$ with the same order and the same degree $r\ge 3$. This approach characterizes a large class of pairs of non-trivial equienergetic regular graphs. Furthermore, it is noted that all the results in this paper are particular cases of Corollary \ref{compl_equi} and Corollary \ref{nec_suf_copl_equ}.	
\end{remark}

\begin{theorem}
	Let $G$ be a graph with  $d_u+d_v\ge 6$ to each edge $e=uv$ in  $G$, then the graphs $\overline{\mathcal{L}^k(G)}$ are hyperenergetic for  $k\ge 2$ if ${\lambda}_{k(1)}\le \frac{n_k-1}{2}$.
\end{theorem}

\begin{proof}
	If $G$ is a graph with  $d_u+d_v\ge 6$ to each edge $e=uv$, then the iterated line graphs $\mathcal{L}^k(G)$  for  $k\ge 2$ are hyperenergetic by Theorem \ref{hyper_en_itLine}, that is $\mathcal{E}({\mathcal{L}^k(G)})\ge 2(n_k-1)$.
	The energy relation \ref{rel_line_cpml} between $\mathcal{E}(\overline{\mathcal{L}^k(G)})$ and $\mathcal{E}({\mathcal{L}^k(G)})$ by Theorem \ref{comple_spec_enrg} is 
	\begin{equation*}
		\mathcal{E}(\overline{\mathcal{L}^k(G)}) =  2\overline{\lambda}_{k(1)}+\frac{\mathcal{E}(\mathcal{L}^k(G))}{2} 
		\ge  2\overline{\lambda}_{k(1)} +\frac{1}{2} 2(n_k-1) 
		= 2\overline{\lambda}_{k(1)} + (n_k-1).
	\end{equation*}
	It is well known that $\lambda_{k(1)}+\overline{\lambda}_{k(1)}\ge n_k-1$, which implies $\overline{\lambda}_{k(1)}\ge n_k-1-\lambda_{k(1)}\ge n_k-1-\frac{n_k-1}{2} = \frac{n_k-1}{2}$ if ${\lambda}_{k(1)}\le \frac{n_k-1}{2}$. By using this we get  $\mathcal{E}(\overline{\mathcal{L}^k(G)})\ge 2\frac{n_k-1}{2}+(n_k-1)=2(n_k-1)$ which completes the proof.
\end{proof}

\section{Other results}
\begin{proposition}\label{indep_num_lin}
	Let $G$ be a graph of order $n_0$ and size $m_0$. If the graphs ${\mathcal{L}^k(G)}$  satisfy  the property $\rho$ with $-2$  multiplicity  $m_{k-1}-n_{k-1}$ for  $k\ge 1$, then $\alpha \big(\mathcal{L}^k(G)\big)\le n_{k-1}$ and $\alpha \big(\overline{\mathcal{L}^k(G)}\big)\le n^-\big(\mathcal{L}^k(G)\big)+1$.
\end{proposition}
\begin{proof}
	If $G$ is a graph of order $n_0$, it is well known that $n^-(G)\le n_0-\alpha(G)$. Using this fact for the graphs ${\mathcal{L}^k(G)}$, we get that $m_{k-1}-n_{k-1} = n_{k}-n_{k-1}\le n_k-\alpha\big(\mathcal{L}^k(G)\big)$ which implies $\alpha \big(\mathcal{L}^k(G)\big)\le n_{k-1}$ for $k\ge 1$. Again using the same fact for the graphs $\overline{\mathcal{L}^k(G)}$, we get that $n_{k-1}-1\le n_k-\alpha\big(\overline{\mathcal{L}^k(G)}\big)$ or $-(n_k-n_{k-1}+1)\le -\alpha\big(\overline{\mathcal{L}^k(G)}\big)$ which gives  $\alpha \big(\overline{\mathcal{L}^k(G)}\big)\le n^-\big(\mathcal{L}^k(G)\big)+1$ for $k\ge 1$.
\end{proof}

\noindent The following result is interesting regarding the minimum order of a graph where both the graph $G$ and its complement are connected and the line graph $\mathcal{L}(G)$ satisfies property $\rho$.

\begin{theorem}\label{min_order_thm}
	The smallest possible order of a connected graph $G$, where its line graph satisfies the property $\rho$ and the complement graph $\overline{G}$	is also connected, is $7$.
\end{theorem}
\begin{proof}
	There are exactly $13$ non-isomorphic connected graphs of order up to 6 whose line graphs satisfy the property \( \rho \). These graphs include $C_4, K_4, K_{3, 2}, K_5, K_{4, 2}, K_{3, 3}, K_6$ and the graphs shown in Figure 2.

	\begin{figure}[h!]
		\centering
		\includegraphics[width=1.0\linewidth]{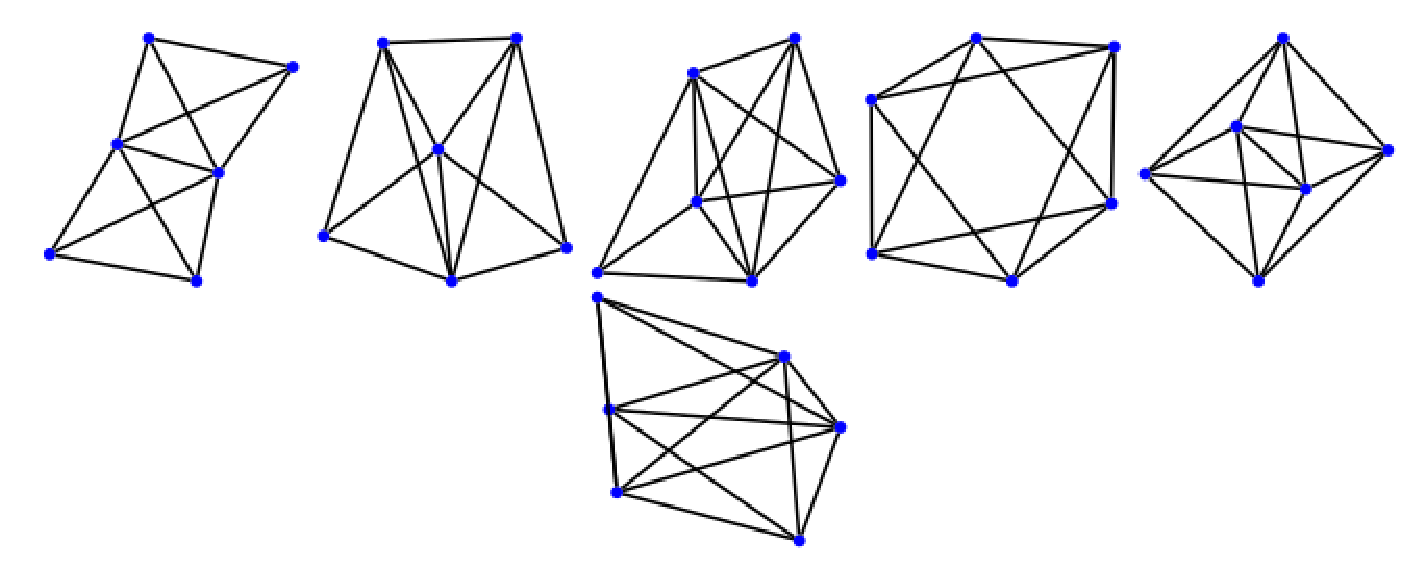}\\
		\label{Fig2}
		\caption{}
	\end{figure}
	
	\vspace{5mm}
	None of the graphs mentioned above have a connected complement. The only connected graph of order $7$ whose line graph satisfies the property $\rho$ and whose complement is also connected is shown in Figure $3$, which completes the proof.
	\begin{figure}[h!]
		\centering
		\includegraphics[width=0.3\linewidth]{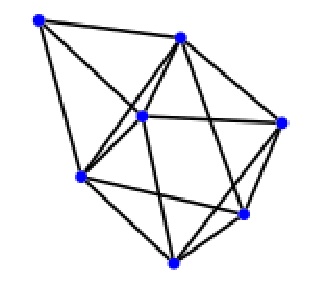}\\
		\label{Fig3}
		\caption{}
	\end{figure}
	
\end{proof}
\section*{Conclusion}\label{sec13}
In this paper, we have described several iterated line graphs $\mathcal{L}^k(G)$ with all negative eigenvalues equal to $-2$ and discussed their energy. We also explored the spectra and the energy of their complements. Additionally, we presented a large class of equienergetic graphs, extending some of the previous results. Although we have identified many classes of iterated line graphs with all negative eigenvalues equal to $-2$, one important question remains: to characterize all graphs with all negative eigenvalues equal to $-2$ which are not necessarily limited to line graphs.\\

\noindent
{\bf Data Availability:}
There is no data associated with this article.\\

\noindent
{\bf Conflicts of interest:} The authors have no conflict of interest.

\end{document}